\documentclass[12pt]{amsart}

\usepackage[all]{xy}
\usepackage{graphics}
\usepackage{color}
\usepackage[T1]{fontenc}
\usepackage{parskip}
\usepackage[colorlinks]{hyperref}
\usepackage{enumitem}

\newcommand{\B}[1]{{\mathbf #1}}

\newtheorem*{theorem}{Theorem}

\newtheorem*{lemma}{Lemma}

\newtheorem{*problem}{Open Problem}
\theoremstyle{definition}

\theoremstyle{remark}

\newcommand{\OP}{\operatorname}

\begin{document}

\title{The autonomous norm on $\OP{Ham}\left( \B R^{2n} \right)$ is bounded}
\author{Michael Brandenbursky}
\address{Ben Gurion University, Israel}
\email{brandens@math.bgu.ac.il}
\author{Jarek K\k{e}dra}
\address{University of Aberdeen and University of Szczecin}
\email{kedra@abdn.ac.uk}


\begin{abstract}
We prove that the autonomous norm on the group of compactly
supported Hamiltonian diffeomorphisms of the standard $\B R^{2n}$
is bounded.
\end{abstract}

\maketitle 

Let $(M,\omega)$ be a symplectic manifold and let $\OP{Ham}(M,\omega)$ be the
group of compactly supported Hamiltonian diffeomorphisms of $(M,\omega)$.
Recall that a Hamiltonian diffeomorphism $f$ is a time-one map of the flow
generated by the vector field $X_{F_t}$ defined by $\omega(X_{F_t},-)=dF_t$.
Here $F\colon M\times S^1\to \B R$ is a smooth compactly supported function
and $F(x,t)=F_t(x)$ (see \cite[Section 5.1]{MR2002g:53157} for details). 
The function $F$ is called a Hamiltonian of $f$. If
$F$ does not depend on the second variable or is time independent then
$f$ is called autonomous. It is known that every Hamiltonian diffeomorphism
is a product of autonomous ones \cite{1602.03287v2}. 
The autonomous norm on $\OP{Ham}(M,\omega)$ is defined by:
$$
\|f\|=\min\{k\in \B N\,|\, f=a_1\cdots a_n, \text{ where $a_i$ is autonomous}\}.
$$
It is a conjugation invariant norm and is known to be
unbounded on the group of compactly supported Hamiltonian
diffeomorphisms of an oriented  surface of finite area
\cite{MR3044593,1405.7931,1602.03287v2,MR2104597}.  

This paper is concerned with the group $\OP{Ham}(\B R^{2n})$ of compactly
supported Hamiltonian diffeomorphism of the Euclidean space equipped with the
standard symplectic form. We prove the following result.  

\begin{theorem}\label{T:main}
The diameter of the autonomous norm on $\OP{Ham}\left( \B R^{2n} \right)$
is bounded above by $3$.
\end{theorem}
\begin{proof}
Let $f\in \OP{Ham}\left( \B R^{2n} \right)$. Let $f=a_1\cdots a_m$, where
$a_i\in \OP{Ham}\left( \B R^{2n} \right)$ are autonomous diffeomorphisms
with compactly supported Hamiltonian functions $F_i\colon \B R^{2n}\to \B R$.
Let $B(r)\subset \B R^{2n}$ be an Euclidean ball of radius $r>0$, centered
at the origin and containing the union of the supports of the functions $F_i$.

\begin{lemma}\label{L:displacement}
There exists an autonomous diffeomorphism $h\in \OP{Ham}\left( \B R^{2n} \right)$
displacing the ball $B(r)$ $m$ times.
\end{lemma}
The statement of the lemma means that $h^i(B(r))\cap h^j(B(r))= \emptyset$
for $0\leq i\neq j\leq m$.
It follows from \cite[Lemma 2.6]{MR2509711} that there exists 
$g\in \OP{Ham}\left( \B R^{2n} \right)$ such that the following
equality holds:
$$
f = a_1\cdots a_m
= [h,g] a_1^h a_2^{h^2}\cdots a_m^{h^m},
$$
where $a_i^{h^i}=h^ia_ih^{-i}$ and $h$ is the diffeomorphism from the Lemma.
Observe that, since the supports of $F_i\circ h^i$ are pairwise disjoint for
$i\in\{1,\ldots,m\}$, we obtain that the composition $a_1^h a_2^{h^2}\cdots a_m^{h^m}$
is autonomous with the Hamiltonian function equal to 
$$
F_1\circ h+ F_2\circ h^2+\cdots + F_m\circ h^m.
$$ 
Since the commutator
$[h,g]=h\cdot h^{g}$ is a product of two autonomous diffeomorphisms
we obtain that $f$ is a product of three autonomous diffeomorphisms.
\end{proof}

\begin{proof}[Proof of the Lemma]
Let $H_1\colon \B R\to \B R$ be a smooth function satisfying the
following conditions:
\begin{enumerate}
\item $H_1(y)=0$ for $|y|>r+1$ 
\item $H_1'(y)=r$ for $|y|\leq r$.
\end{enumerate}

Let $H(x_1,y_1,\ldots,x_n,y_n) = H_1(y_1)$. We have that $dH = rdx_1$
and that the induced Hamiltonian vector field $X$ is equal to 
$r\frac{\partial}{\partial x_1}$. Thus the induced Hamiltonian diffeomorphism 
displaces the
ball $B(r)$ as many times as we like. Taking an appropriate cut off
function we obtain the required compactly supported diffeomorphism $h$.
\end{proof}

\subsection*{Remarks}
If $f$ in Theorem \ref{T:main} is contained in the kernel of the Calabi homomorphism
(see Section 8.B of \cite{MR1612569} for a definition) then the same argument
shows that it is a product of up to three autonomous diffeomorphisms with
trivial Calabi invariant.

It is known that the Hofer norm on $\OP{Ham}\left( \B R^{2n} \right)$ is
unbounded and stably bounded \cite{MR96g:58001}. The kernel of the Calabi homomorphism
does not admit nontrivial quasimorphisms, however, it is stably unbounded \cite{kawasaki}. 

It is not difficult to see that the diameter of the autonomous norm on $\OP{Ham}\left( \B R^{2n} \right)$ is at least $2$. 
To the best of our knowledge it is an open question whether there exists a Hamiltonian diffeomorphism of $ \B R^{2n}$
of autonomous norm equal to $3$.

\subsection*{Acknowledgements}
We thank the Center for Advanced Studies in Mathematics at Ben Gurion University
for supporting the visit of the second author at BGU.

\bibliography{bibliography}

\def\polhk#1{\setbox0=\hbox{#1}{\ooalign{\hidewidth
  \lower1.5ex\hbox{`}\hidewidth\crcr\unhbox0}}}
  \def\polhk#1{\setbox0=\hbox{#1}{\ooalign{\hidewidth
  \lower1.5ex\hbox{`}\hidewidth\crcr\unhbox0}}}
  \def\polhk#1{\setbox0=\hbox{#1}{\ooalign{\hidewidth
  \lower1.5ex\hbox{`}\hidewidth\crcr\unhbox0}}}
  \def\polhk#1{\setbox0=\hbox{#1}{\ooalign{\hidewidth
  \lower1.5ex\hbox{`}\hidewidth\crcr\unhbox0}}}
  \def\polhk#1{\setbox0=\hbox{#1}{\ooalign{\hidewidth
  \lower1.5ex\hbox{`}\hidewidth\crcr\unhbox0}}}
  \def\polhk#1{\setbox0=\hbox{#1}{\ooalign{\hidewidth
  \lower1.5ex\hbox{`}\hidewidth\crcr\unhbox0}}}
\begin{thebibliography}{1}

\bibitem{MR1612569}
Vladimir~I. Arnold and Boris~A. Khesin.
\newblock {\em Topological methods in hydrodynamics}, volume 125 of {\em
  Applied Mathematical Sciences}.
\newblock Springer-Verlag, New York, 1998.

\bibitem{MR3044593}
Michael Brandenbursky and Jarek K{\polhk{e}}dra.
\newblock On the autonomous metric on the group of area-preserving
  diffeomorphisms of the 2-disc.
\newblock {\em Algebr. Geom. Topol.}, 13(2):795--816, 2013.

\bibitem{1602.03287v2}
Michael Brandenbursky, Jarek Kedra, and Egor Shelukhin.
\newblock On the autonomous norm on the group of {H}amiltonian diffeomorphisms
  of the torus.
\newblock Available at \url{http://arxiv.org/abs/1602.03287v2}.

\bibitem{1405.7931}
Michael Brandenbursky and Egor Shelukhin.
\newblock On the {$L^p$}-geometry of autonomous {H}amiltonian diffeomorphisms
  of surfaces.
\newblock {\em Math. Res. Lett. to appear}.
\newblock Available at \url{http://arxiv.org/abs/1405.7931v2}.

\bibitem{MR2509711}
Dmitri Burago, Sergei Ivanov, and Leonid Polterovich.
\newblock Conjugation-invariant norms on groups of geometric origin.
\newblock In {\em Groups of diffeomorphisms}, volume~52 of {\em Adv. Stud. Pure
  Math.}, pages 221--250. Math. Soc. Japan, Tokyo, 2008.

\bibitem{MR2104597}
Jean-Marc Gambaudo and {\'E}tienne Ghys.
\newblock Commutators and diffeomorphisms of surfaces.
\newblock {\em Ergodic Theory Dynam. Systems}, 24(5):1591--1617, 2004.

\bibitem{MR96g:58001}
Helmut Hofer and Eduard Zehnder.
\newblock {\em Symplectic invariants and {H}amiltonian dynamics}.
\newblock Birkh\"auser Advanced Texts: Basler Lehrb\"ucher. [Birkh\"auser
  Advanced Texts: Basel Textbooks]. Birkh\"auser Verlag, Basel, 1994.

\bibitem{kawasaki}
Morimichi Kawasaki.
\newblock Relative quasimorphisms and stably unbounded norms on the group of
  symplectomorphisms of the euclidean spaces.
\newblock {\em Journal of Symplectic Geometry}, 2014.

\bibitem{MR2002g:53157}
Leonid Polterovich.
\newblock {\em The geometry of the group of symplectic diffeomorphisms}.
\newblock Lectures in Mathematics ETH Z\"urich. Birkh\"auser Verlag, Basel,
  2001.

\end{thebibliography}
\bibliographystyle{plain}

\end{document}